\documentclass[a4paper,12pt]{article}
\usepackage{localeng}
\usepackage{tikz}
\usetikzlibrary{calc}
\DeclareMathOperator{\Var}{\textrm{Var}}
\newtheorem{theorem}{Theorem}
\begin{document}
\title{Bishop's (up)crossing inequality and lower semicomputable random reals revisited}
\author{Mikhail Andreev\thanks{CRITEO, \texttt{mikhail.andreev.57@gmail.com}}, Alexander Shen\thanks{LIRMM, Univ Montpellier, CNRS, Montpellier, France, \thanks{sasha.shen@gmail.com, alexander.shen@lirmm.fr}. Supported by ANR-21-CE48-0023 FLITTLA and  ANR-24-CE48-3758 CADO grants. Part of the work was done while visiting Neapolis University, Paphos, Cyprus}}

\maketitle

\begin{abstract}
In this paper we provide an easy proof of Barmpalias--Lewis-Pye result saying that all computable increasing sequences converging to random reals converge at the same speed by noting that it easily follows from Bishop's upcrossing inequality. We also provide a simple derivation of this inequality.
\end{abstract}
\bigskip

\begin{flushright}
\emph{To the memory of Vladimir V'yugin}
\end{flushright}

\section{Crossing inequalities}

We start by a symmetric version of Bishop's upcrossing inequality~\cite{bishop1966} (so we name it `crossing inequality' instead of `upcrossing' or `downcrossing' one).

Consider the following ``anti-slalom'' game. A player starts at some point $(x,y)$ and moves to the right. She wants to \emph{miss} as many gates (vertical segments $[m_i,M_i]$ located on the vertical line with first coordinate $x_i$) as possible. The restriction is that the slope of her trajectory should be between $\alpha$ and $\beta$ all the time.
\medskip

\begin{center}
\includegraphics{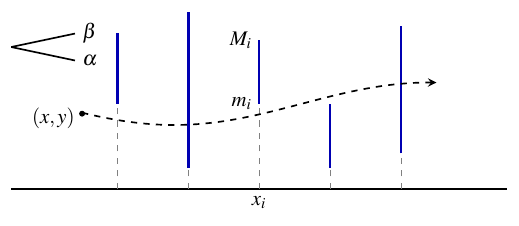}
\end{center}

The game is called \emph{anti-slalom} since the player wants to miss the gates as much as possible. If there were no restrictions on the slope ($\alpha=-\infty$, $\beta=\infty$), she could avoid all the gates. The smaller is the interval $[\alpha,\beta]$, the more restricted is the movement. When the interval is very small, the player essentially moves along some straight line, and crosses all the gates that intersect this straight line.

Fix the positions of the gates (i.e., $x_i$, $m_i$ and $M_i)$ as well as $\alpha<\beta$. Define the function $t(x,y)$ that is the minimal number of crossings if player starts at the point $(x,y)$.  Then consider the function $T(x)=\max_y t(x,y)$.

For example, if $x$ is greater than all $x_i$, we have $T(x)=0$ since there are no gates on the right of $x_i$. We also have $T(x)=0$ if $x$ is much smaller than all $x_i$: if we start far away, small changes in the slope (that should be between~$\alpha$ and $\beta$) are enough to avoid all the gates. So $T(x)$ is a function with finite support (and integer values).

\begin{theorem}[Crossing inequality]\label{cross}
\[
 \int_{\mathbb{R}} T(x)\, dx \le \frac{1}{\beta-\alpha}\sum_i (M_i-m_i).
\]
\end{theorem}

For example, if there is only one gate, then $t(x,y)=1$ inside the triangle (see the picture) and $0$ elsewhere, so $T$ equals $1$ on an interval of length $l/(\beta-\alpha)$ where $l$ is the length of the gate, and the inequality we claim becomes an equality.
\begin{center}
\includegraphics{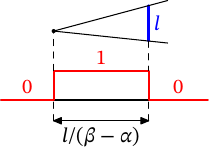}
\end{center}

If there are two gates, the situation is more complicated; for example, on the left picture the inequality is still an equality (since two marked intervals are equal), while on the right picture the inequality is strict.
\begin{center}
\includegraphics[width=0.46\textwidth]{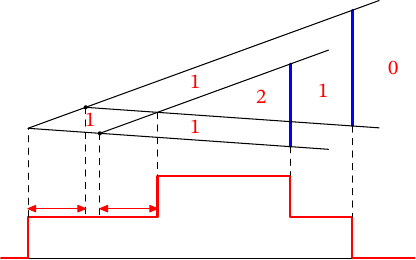}\quad
\includegraphics[width=0.33\textwidth]{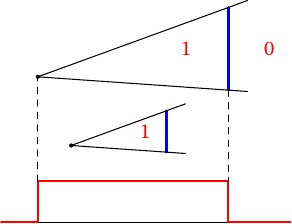}
\end{center}

\begin{proof}[Proof of Theorem~\ref{cross}]
To understand the behavior of the function $t$, we use the moving line approach. For every $x$ we consider the function $y\mapsto t(x,y)$ and look how this function changes when $x$ decreases.

\begin{itemize}
\item When $x$ is greater than all $x_i$ (positions of all gates), the function is zero everywhere.
\item When $x$ crosses the coordinate $x_i$ of $i$th gate (changing from $x_i+0$ to $x_i-0$, so to say), the indicator function of the interval $[m_i,M_i]$ is added (new gate crossing).
\item When $x$ decreases from $x+\Delta x$ to $x$ not crossing any of $x_i$, the function $t(x,y)$ can be computed as follows:
\[
t(x,y)=\min_{\Delta y \in [\alpha\Delta x,\beta\Delta x]} t(x+\Delta x,y+\Delta y).
\]
Indeed, the player could move (without crossing any gates) from $(x,y)$ to the point $(x+\Delta x,y+\Delta y)$ for arbitrary $\Delta y \in [\alpha\Delta x,\beta\Delta x]$. She needs to choose the point that minimizes the number of subsequent crossings.
\end{itemize}

The last formula means that the widths of the peaks of the function $t(x,\cdot)$ (marked by arrows in the diagram)
\begin{figure}[h]
$$\text{\includegraphics{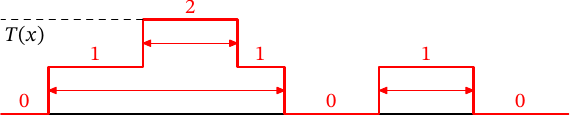}}$$
\end{figure}
decrease with speed $\beta-\alpha$ when $x$ decreases (and does not cross any $x_i$). Therefore, the integral $I(x)=\int_y t(x,y)\,dy$ of this function decreases with speed at least $T(x)(\beta-\alpha)$. 

The total decrease of the integral cannot exceed the sum of its increases (that happen due to gates), therefore
\[
(\beta-\alpha)\int_x T(x) \le \sum_i (M_i-m_i),
\]
as required.
\end{proof}

\section {Crossing a curve}

Now we want to get an upper bound for the number of crossings for the case where we cross some (polygonal) line instead of a set of vertical gates.

A simple trick reduces this case to the case of vertical gates. Indeed, if the trajectory with slope in $[\alpha,\beta]$ crosses a non-vertical segment $AB$ (with horizontal size $\Delta x$ and vertical size $\Delta y$), then it has to cross also the vertical segment $A'B$ that is a union of all intervals reachable (with slope in $[\alpha,\beta]$) from all points between $A$ and $B$.
\begin{center}
\includegraphics{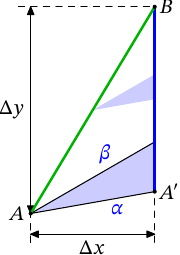}
\end{center}
 What is the length of this new interval? In our example $|A'B|=\Delta y - \alpha \Delta x$. In the general case, the length of the resulting vertical segment depends on the slope of $AB$ (whether it is less than $\alpha$, between $\alpha$ and $\beta$, or greater than $\beta$). Assuming $\Delta x>0$ (while $\Delta y$ can be positive or negative), we get the following answer for the length of the replacement vertical segment:
 \[
 \text{vertical length}=
 \begin{cases}
 \text{$\Delta y - \alpha\Delta x$, if $\Delta y/\Delta x \ge \beta$;}\\
 \text{$(\beta-\alpha)\Delta x$, if $\alpha \le \Delta y/\Delta x \le \beta$;}\\
 \text{$\beta\Delta x - \Delta y$, if $\Delta y/\Delta x \le \alpha$.}
 \end{cases} 
 \]
Note that the dependence on $\Delta y$ is continuous, and any of the two expressions can be used in border cases. Another way to say the same thing: the length of the replacement vertical segment is $\tau_{\alpha,\beta}(\Delta y/\Delta x)\Delta x$, where  $\tau_{\alpha,\beta}$ is the following function:
\begin{center}
\includegraphics{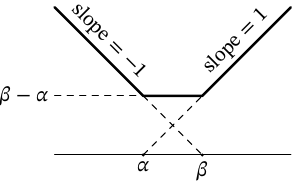}
\end{center}

Now we may consider a polygonal curve $y=y(x)$ for $x\in [l,r]$ as a union of segments, and define $t(x,y)$ as the minimal number of intersections with the curve  for a player who starts at $x,y$ and keeps her slope in $[\alpha,\beta]$. Then $T(x)$ is defined as $\max_y t(x,y)$ (as before). Replacing each segment by a vertical one as described, we get the following bound:
\begin{theorem}[Curve crossing inequality]\label{curve-crossing}
\[
 \int_{\mathbb{R}} T(x)\, dx \le \frac{1}{\beta-\alpha}\int_l^r \tau_{\alpha,\beta}(y'(x))\, dx.
\]
\end{theorem}
Our argument applies directly to polygonal curves; we may use approximations to extend it to any curve that is sufficiently well-behaved, but this is not needed for the sequel.

If $\alpha$ and $\beta$ are close to zero, the function $\tau_{\alpha,\beta}$ is close to the absolute value function, so the right hand size is close to the variation of the function $y$. If $\alpha$ is close to $\beta$ (but they are not necessarily close to $0$), the right hand side is close to the variation of the function $y(x)-\alpha x$ (or $y(x)-\beta x$).

\section{Crossing a gap}\label{sec:gap-crossing}

We switched to curves because in this way we can get a bound for the up- and downcrossings. Let us explain what it means. Consider again some curve~$C$ (again we assume that $C$ is a graph of a piecewise linear function, this is enough for us), some point $(x,y)$ and slopes $\alpha<\beta$. We count how many times the curve crosses up and down the gap between slopes $\alpha$ and $\beta$, i.e., an angle formed by $\alpha$- and $\beta$-rays starting at $(x,y)$.
\begin{center}
\includegraphics{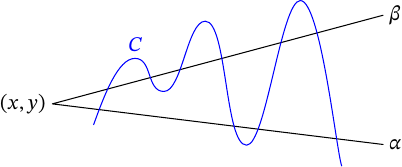}
\end{center}
In other words, we are in $(x,y)$ and watch the movement of a point along $C$ (from left to right). In this example (see picture) we see that initially the point is at an angle smaller than $\alpha$, then the angle becomes greater than $\beta$, then it goes down below $\beta$ but does not reach $\alpha$ and turns back exceeding $\beta$ again, then becomes less than $\alpha$, than again greater than $\beta$ and finally less than $\alpha$. We say that there are two \emph{upcrossings}  and two \emph{downcrossings} of the same gap (between $\alpha$ and $\beta$) in this picture.  The numbers of up- and downcrossings differ at most by~$1$ (since they alternate), and we consider the sum of them.

More formally, we consider sequences of points $c_0,\ldots,c_m$ on the curve (going from left to right) that alternate being below $\alpha$-ray and above $\beta$-ray, and take the maximal $m$ in all those sequences. This maximal $m$ is the number of crossings for the gap and is denoted $t(x,y)$. So for every fixed curve that is the graph of the function $y(x)$ for $x\in [a,b]$, and for every fixed $\alpha<\beta$ we get some non-negative integer function $t(x,y)$ that depends on the starting point $(x,y)$. Then we let $T(x)=\max_y t(x,y)$.

\begin{theorem}[Bishop's crossings inequality]\label{bishop}
\[
 \int_{\mathbb{R}} T(x)\, dx \le \frac{1}{\beta-\alpha}\int_a^b \tau_{\alpha,\beta}(y'(x))\, dx.
\]
\end{theorem}

This inequality is the same as in Theorem~\ref{curve-crossing}, and this result is an immediate corollary of that theorem. Indeed, any trajectory that goes from $(x,y)$ with slope between $\alpha$ and $\beta$, is confined in the gap and therefore has at least one intersection point for each part of the curve that crosses the gap.

Note that $t_{\alpha,\beta}(s) \le |s| +O(|\alpha|+|\beta|)$, and the integral of $|y'|$ is the variation of the function $y(x)$, therefore we get
\[
 \int_{\mathbb{R}} T(x)\, dx \le \frac{O(|\alpha|+|\beta|)(b-a)+\Var(y)}{\beta-\alpha}\eqno(*)
\]
and this will be enough for our application.

\section{Barmpalias--Lewis-Pye theorem}

The Bishop's crossing inequality implies (almost immediately) the following result of Barmpalias and Lewis-Pye~\cite[Theorem 1.7]{blp2017}; see also~\cite{miller2017}:

\begin{theorem}[Barmpalias, Lewis-Pye]\label{blp}
Let $a_n$ and $b_n$ be two increasing computable sequences of rational numbers that converge to real numbers $A$ and $B$. Assume that $A$ is Martin-L\"of random. Then the sequence
\[
\frac{B-b_n}{A-a_n}
\]
has a limit.
\end{theorem}

(See, e.g., \cite{usv2018} for the definition of Martin-L\"of randomness and basic results about monotone converging sequences of rational numbers and their limits, lower semicomputable reals.)

The argument below proves that there exists a finite or infinite limit, but in fact the second case is impossible (and this was known much before the result of  Barmpalias and Lewis-Pye, see~\cite{KuceraSlaman2001}; we reproduce the argument below in Section~\ref{history} for reader's convenience).

\begin{proof}[of Theorem~\ref{blp}]
If the sequence $(B-b_n)/(A-a_n)$ has no finite or infinite limit, there is a gap $\alpha<\beta$ that is crossed infinitely many times (there are infinitely many terms strictly below $\alpha$ and infinitely many terms strictly above $\beta$).   We need to show that $A$ is non-random.

Consider the curve connecting $(a_1,b_1)$, $(a_2,b_2),\ldots$ and converging to $(A,B)$. 
\begin{center}
\includegraphics[width=0.5\textwidth]{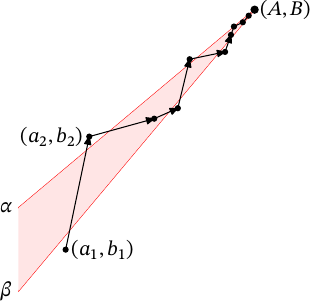}
\end{center}
We know that this curve has infinitely many crossings with the angle that has vertex $(A,B)$ and slopes $\alpha$ and $\beta$. Now we can apply Bishop's inequality (rotated by $180^\circ$): let $t(a,b)$ be number of crossings, considered as a function of the angle vertex $(a,b)$ and let $T(a)$ be the maximum of $t(a,b)$ over all $b$. Then $T(A)=+\infty$ while $\int_a T(a)$ is finite (the curve is monotone and bounded, so the variation is finite). It remains to note that function $T$ is lower semicomputable (for the finite part of the curve it is computable, and it increases when we add new segments to the curve), so we get an integral test that rejects $A$.
\end{proof}

Formally speaking, we proved Theorem~\ref{bishop} for finite polygonal curves and now we have infinite one. To avoid problems, we consider the bound for the finite part of the curve and note that all integrals are uniformly bounded. Then we apply the monotone convergence theorem.

\section{Remarks}\label{history}

\subsection*{$(B-b_n)/(A-a_n)$ is bounded}

For reader's convenience, let us provide a proof of a classical result from~\cite{KuceraSlaman2001} saying that $(B-b_n)/(A-a_n)$ is bounded (under the assumptions of Theorem~\ref{blp}). Assume that for every $\varepsilon>0$ there exists some $n$ such that $A-a_n < \varepsilon(B-b_n)$. Let us show how to get an effectively open cover for $A$ of arbitrarily small measure. 
\begin{center}
\includegraphics{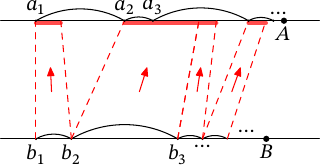}
\end{center}
For every interval $(b_i,b_{i+1})$ on the bottom line consider its ``counterpart'' on the top line, an interval whose length is multiplied by $\varepsilon$, i.e., equals $\varepsilon(b_{i+1}-b_i)$,  that starts from $b_i$ or from the right endpoint of the counterpart of the previous interval, \emph{whichever is bigger}.

The total length of all counterparts is at most $\varepsilon(B-b_1)$, and can be made arbitrarily small if $\varepsilon$ is small enough. On the other hand, they cover $A$ (so $A$ is non-random). Why? Since $A-a_n < \varepsilon(B-b_n)$ for some $n$, the counterparts of $(b_n,b_{n+1})$, $(b_{n+1}, b_{n+2}),\ldots$ cannot be all below $A$ (note that they are disjoint). So some of them should go above $A$. On the other hand, $A$ cannot get into the gaps between the counterparts, since all gaps are below some $a_i$, and $A$ is greater than all $a_i$. This finishes the argument.

\subsection*{Zero limit and non-random $B$}

If $B$ is random, the limit of the inverse fraction $(A-a_n)/(B-b_n)$ exists (and is finite, as we have seen), therefore both limits are non-zero. Therefore, if the limit $(B-b_n)/(A-a_n)$ is zero, then $B$ is non-random.

On the other hand, if $a_n$ and $b_n$ are computable increasing sequences converging to $A$ and $B$, and $A-a_n=\Theta(B-b_n)$ (i.e., $A-a_n$ and $B-b_n$ differ at most by some constant factor), then $A$ and $B$ are either both random or both non-random. This is another classical result that goes back to Solovay (see~\cite[Theorem 4.9]{CHKW2001}). The direct argument may go as follows. Let us assume that $A$ is non-random; we need to show that $B$ is also non-random. Assume that we have an enumeration of intervals that together cover $A$ and have small total measure. How can we cover $B$ by some other enumerable family of intervals with small measure? When some interval $(p,q)$ appears in the cover of $A$, we can wait until some $a_i$ turns out to be in this interval. (If this never happens, the interval does not cover $A$ and is useless.) If some $a_i$ appears in $(p,q)$, we take an interval of $c$ times bigger length, where $c$ is the constant in $B-b_i=O(A-a_i)$, with left endpoint $b_i$. We do not know whether $(p,q)$ covers $A$, but if it does, then the new interval covers $B$ (due to the choice of $c$). Since some interval in the enumeration covers $A$ by assumption, the new family covers $B$ and it at most $c$ times longer. This shows that if $B-b_i=O(A-a_i)$ and $A$ is non-random, then $B$ is non-random, too. 

These arguments show that the limit in Theorem~\ref{blp} is always finite and equals zero if and only if $B$ is non-random.

\subsection*{Only  $A$ and $B$ matter}

One can note that the limit in Theorem~\ref{blp} does not depend on the choice of computable increasing sequences that converge to $A$ and $B$ (and depends on $A$ and $B$ only). Indeed, for non-random $B$ this limit is always $0$. For random $B$ it is enough to show that two increasing sequences $a_n$ and $a_n'$ that converge to the same random $A$, have the same convergence speed, i.e. $(A-a_n)/(A-a_n')$ converges to $1$.  Indeed, we know already that the limit exists, assume that it is not $1$. Then one sequence converges faster than the other; we may assume without loss of generality that $A-a_n' < c(A-a_n)$ for some $c<1$ and for all $n$ (discarding some initial parts of the sequences). This implies that $A$ is computable. Indeed, if we know some approximation $a$ to $A$ from below, we can wait until $a_i$ becomes greater than $a$. Then $a_i'$ is a better approximation to $A$ that $a$ (by factor $c$). Repeating this trick, we can compute $A$ with arbitrary precision.

\subsection*{Extensions: arbitrary random reals, finite variation functions}

Ivan Titov~\cite{titov2024} proved the following generalization of the Barmpalias--Lewis-Pye result where $A$ is not assumed to be a left-c.e. real:

\begin{theorem}[Titov]
Let $A$ be a Martin-L\"of random real, and let $g$ be a partial computable non-decreasing bounded function that is defined on all rational numbers smaller than $A$ \textup(and, maybe, on some other rational numbers\textup). Let $\displaystyle B=\lim_{r\to A-0} g(r)$. Then the ratio
\[
\frac{B-g(r)}{A-r} 
\]
has limit as $r\to A-0$.
\end{theorem}

This result can be obtained as a corollary of crossing inequalities in the same way. The only difference is here we have to deal not with a polygonal line, but with a graph of a partial function defined on a countable set.  But the definitions of the number of crossings and the variation can be naturally extended to this case, and our main tool, the inequality $(*)$ (see the end of Section~\ref{sec:gap-crossing}) remain valid for this case. Indeed, it is true for a function with finite domain (we can connect the dots by a polygonal line), and then we can take the limit and prove it for countable domains. 

Note that in this way we prove only that finite or infinite limit exists (and to exclude an infinite limit another argument is needed; it is provided in~\cite{titov2024}). 

This argument extends (without any changes) to functions $g$ that have finite variation (but are not necessarily monotone) --- a result proven by Titov in another paper~\cite[Theorem 20]{titov2025}.

\subsection*{Historical remarks}

Two results (of Bishop and Barmpalias--Lewis-Pye) we are reproving here have a quite complicated history.
The upcrossing inequality appeared in Bishop's paper~\cite{bishop1966} in slightly different, but essentially equivalent form; Lemma~1 in this paper corresponds to our Theorem~\ref{cross}, and Theorem~1 in Bishop's paper roughly corresponds to Theorem~\ref{bishop}. The main technical innovation in our proof is to introduce the ``anti-slalom'' game that provides an upper bound for the number of upcrossings and is easier to analyze than the number of upcrossings itself.

Bishop used these results to provide a new proof of Birkhoff ergodic theorem (a more constructive one --- he was interesting in constructive versions of different mathematical results, and this was his main motivation to revisit the ergodic theorem).  In his book~\cite{bishop1966} he reproves the constructive ergodic theorem, this time essentially using some other (and somehow weaker) version of the upcrossing inequality. The difference is that in this weaker version we consider not the maximum of $t(x,y)$ for given $x$ and for all $y$, but consider only points $(x,y)$ on the curve. This is enough for Bishop's proof of ergodic theorem, but not for our application. Finally, Bishop~\cite{bishop1968} revisits ergodic theorem in more general setting, but does not state explicitly any version of upcrossing inequality for plane curves.

The upcrossing inequalities already appeared in algorithmic randomness theory for a completely different reason: V'yugin used them in~\cite{vyugin1998} to prove the ergodic theorem for Martin-L\"of random sequences (under some computability assumptions on the measure). For this he used the weaker version of upcrossing inequality from~\cite{bishop1967}. In fact, neither Bishop in~\cite{bishop1967} nor V'yugin in~\cite{vyugin1998} explicitly use this weaker version of the upcrossing inequality; instead, they incorporate its proof in the proof of the ergodic theorem, but all the tools are there. See~\cite{shen-ergodic} for the details (the weak version of the inequality appears there as Lemma~1, and the Bishop--V'yugin potential argument is reproduced in Section~3). Note also that a bound for crossings appears also in the proof of the martingale convergence theorem.

Theorem~\ref{blp} was proven by Barmpalias and Lewis-Pye~\cite{blp2017} (without any connection to ergodic theorem for Martin-L\"of randomness or upcrossing inequalities). It appears in their paper as Theorem~1.7, and the crucial step in the proof is Lemma~2.1 that says essentially that the number of $(\alpha,\beta)$-crossings is finite for every gap $(\alpha,\beta)$. This Lemma is proven in Section~3 using quite complicated construction with movable markers.

Another proof of Theorem~\ref{blp} was given by Miller~\cite{miller2017}; as the author states, one of the main purposes of his note is ``to give relatively short, self-contained proofs of the results of Barmpalias and Lewis-Pye''. The statement of Theorem~\ref{blp} appears as Lemma~1.2 there, and the proof is relatively short (though a bit technical). Of course, this is a matter of taste, but we hope that our proofs are more intuitive and maybe even can be considered as ``proofs from The Book''. The other results mentioned in this section are also well known; we included them only to provide easy direct proofs.

\section*{Acknowledgements}

The authors thank Ivan Titov, Laurent Bienvenu and Wolfgang Merkle who tried hard to explain to us the results about convergence speed for lower semicomputable reals. Ivan Titov kindly informed us that some of his results also follow from crossing inequalities. We also thank Ilya Brodsky without whom this paper could not appear. We are grateful to the anonymous referees of our submission to CiE 2026 conference whose critical remarks helped us to provide more historical context and make clear that we just reproduce the proofs of known results for readers' convenience (and to correct language errors).

During the preparation of this paper one of us (A.~Shen) was supported by ANR-21-CE48-0023 FLITTLA  and  ANR-24-CE48-3758 CADO grants. Part of the work was done while visiting Neapolis University, Paphos, Cyprus.

This paper is dedicated to the memory of our late colleague Vladimir V'yugin, a great mathematician and a kind man, who first used Bishop's upcrossing inequalities in the algorithmic randomness theory.

\end{document}